\documentclass[12pt]{amsart}
\usepackage{graphicx}
\usepackage{amsmath}
\usepackage{amssymb, amscd}
\usepackage{epstopdf}
\usepackage[all]{xy}
\newtheorem{theorem}{Theorem}[section]
\newtheorem{lemma}[theorem]{Lemma}
\newtheorem{corollary}[theorem]{Corollary}

\newtheorem{prop}[theorem]{Proposition}
\theoremstyle{definition}

\def\eproof{$\Box$ \medskip}

\def\chb{{\rm Dome}(\Omega)}
\def\chbz{{\rm Dome}(\Omega_0)}
\def\chbt{{\rm Dome}(\Omega_t)}
\def\chbu{\widetilde{{\rm Dome}(\Omega)}}

\newcommand\rs{\hat{\mathbb{C}}}
\newcommand\C{\mathbb{C}}

\newcommand\R{\mathbb{R}}
\newcommand\Sp{\mathbb{S}}
\newcommand\U{\mathbb{U}}
\newcommand\Ht{\mathbb{H}^3}
\newcommand\Hp{\mathbb{H}^2}

\newcommand\Z{\mathbb{Z}}

\title[Uniformly perfect domains and convex hulls]{Uniformly perfect domains and convex hulls: improved bounds in a
generalization of a Theorem of Sullivan}

\author{Martin Bridgeman}
\address{Boston College}
\author{Richard D. Canary}
\address{University of Michigan}
\dedicatory{Dedicated to Dennis Sullivan on the occasion of his 70th birthday.}
\date{\today}
\thanks{Canary was partially supported by NSF grant DMS - 1006298.}

\begin{document}
\begin{abstract}
Given a hyperbolic domain $\Omega$, the nearest point retraction is
a conformally natural homotopy equivalence from $\Omega$ to  the
boundary $\chb$ of the convex core of its complement. Marden and Markovic
showed that if $\Omega$ is uniformly perfect, then there exists a conformally natural
quasiconformal map from $\Omega$ to $\chb$ which admits a bounded homotopy
to the nearest point retraction. We obtain an explicit upper bound on the quasiconformal
dilatation which depends
only on the injectivity radius of the domain.  \end{abstract}

\maketitle

\section{Introduction}

If $\Omega$ is a hyperbolic domain in the Riemann sphere $\rs$, then the boundary
$\chb$ of the convex core of its complement is a hyperbolic surface in its
intrinsic metric  (\cite{EM87,ThBook}). The nearest point retraction $r:\Omega\to\chb$,
which maps
a point $z\in\Omega$ to the unique point of intersection of the smallest horoball based at $z$ which intersects
$\chb$, gives a conformally natural homotopy equivalence.
Sullivan \cite{sullivan} (see also \cite{EM87,EMM2}) showed that there exists $K_0$ such that
if $\Omega$ is simply connected, then $r$ is homotopic to a  conformally natural
$K_0$-quasiconformal map.
Marden and Markovic \cite{MM} showed that if $\Omega$ is uniformly perfect, then
there is  a  conformally natural quasiconformal map from $\Omega$ to $\chb$ which admits
a bounded homotopy to the nearest point retraction and that the dilatation of
the map may be bounded in terms of  the injectivity
radius of $\Omega$ in the Poincar\'e metric.  (We recall that $\Omega$ is uniformly
perfect if and only if there is a non-zero lower bound for the injectivity radius of $\Omega$ in
the Poincar\'e metric.) In this paper, we combine the results of \cite{BC10} and  the techniques of
Epstein-Marden-Markovic  \cite{EMM2} to obtain explicit  bounds on the quasiconformal
dilatation. We further show that these bounds are close to optimal.

\begin{theorem}
There exists a function $M: \R_+ \rightarrow \R_+$ such that if
$\Omega\subset\rs$ is uniformly perfect and
$\nu>0$ is a lower bound
for its injectivity radius in the Poincar\'e metric, then there is a  conformally natural
\hbox{$M(\nu)$-quasiconformal} map 
$$\phi:\Omega\to\chb$$
which admits
a bounded homotopy to the nearest point retraction $r:\Omega\to\chb$.
Moreover,
$$ M(\nu)=48\pi e^me^{\pi^2/2\nu}+12\pi\le 2220e^{\pi^2/2\nu}+ 38$$
where $m=\cosh^{-1}(e^2)\approx 2.689$.
\label{main}
\end{theorem}

We also obtain a version of our main theorem, where the quasiconformal dilatation is bounded by
the injectivity radius of $\chb$. We recall that $\Omega$ is uniformly perfect if and only if
there is a positive lower bound for the injectivity radius of $\chb$  (see Bridgeman-Canary
\cite[Lemmas 8.1 and 9.1]{BC03}).

\begin{theorem}
\label{main2}
There exists a function $N: \R_+ \rightarrow \R_+$ such that if
$\Omega\subset\rs$ is uniformly perfect and
$\hat\nu>0$ is a lower bound
for the injectivity radius  of $\chb$ in its intrinsic metric, then there is a  conformally natural
\hbox{$N(\hat\nu)$-quasiconformal} map 
$$\phi:\Omega\to\chb$$
which admits a bounded homotopy to the nearest point retraction \hbox{$r:\Omega\to\chb$}.
Moreover,
$$ N(\hat\nu) =\frac{24\pi}{\hat\nu}+12\pi .$$
\end{theorem}

We recall that a map $f:\Omega\to \Ht$ is said to be {\em conformally natural} if
whenever $\Gamma$ is a group of conformal automorphisms  of $\rs$ which preserve $\Omega$,
then 
$$f(\gamma(z))=\gamma(f(z))$$ 
for all $z\in\Omega$ and $\gamma\in \Gamma$ where on the right-hand side 
we have extended the conformal automorphism of $\rs$ to the associated isometry of
$\Ht$.

If $N=\Ht/\Gamma$ is a hyperbolic 3-manifold, we let $\Omega(\Gamma)$ be the largest
open set in $\rs$ which $\Gamma$ acts properly discontinuously on and consider
the {\em conformal boundary} \hbox{$\partial_cN=\Omega(\Gamma)/\Gamma$}. 
We will assume that $\Gamma$ is not virtually abelian, which guarantees that $\Omega(\Gamma)$ is
hyperbolic. The 
{\em convex core} $C(N)$ of $N$ is  the quotient
$CH(\rs-\Omega(\Gamma))/\Gamma$ where $CH(\rs-\Omega(\Gamma))$ denotes
the convex hull in $\Ht$ of $\rs-\Omega(\Gamma)$. Then
$${\rm Dome}(\Omega(\Gamma))=\partial CH(\rs-\Omega(\Gamma))$$
and the
nearest point retraction $r:\Omega(\Gamma)\to {\rm Dome}(\Omega(\Gamma))$ descends to 
a homotopy equivalence 
$$\bar r:\partial_cN\to \partial C(N).$$
If $\Omega(\Gamma)$ is uniformly perfect, 
the conformally natural quasiconformal map
$$\phi:\Omega(\Gamma)\to {\rm Dome}(\Omega(\Gamma))$$
guaranteed by Theorem \ref{main}   descends to a quasiconformal map
$$\bar\phi:\partial_cN\to \partial C(N)$$
which admits a bounded homotopy to $\bar r$.
We recall that $\Omega(\Gamma)$ is uniformly
perfect whenever $\Gamma$ is finitely generated (see Pommerenke \cite{pommerenke}).

\begin{corollary}
If $N=\Ht/\Gamma$ is a hyperbolic 3-manifold and
$\nu>0$ is a lower bound for the injectivity radius of $\Omega(\Gamma)$,
then there is a
$M(\nu)$-quasiconformal map 
$$\bar\phi:\partial_cN\to\partial C(N)$$
which admits
a bounded homotopy
to the nearest point retraction $\bar r$.

Moreover, if $\hat{\nu}>0$ is a lower bound for the injectivity radius of 
${\rm Dome}(\Omega(\Gamma))$, then $\phi$ is $N(\hat{\nu})$-quasiconformal.
\end{corollary}

It is instructive to consider the special case of
a round annulus  $\Omega(s)$ of modulus $\frac{s}{2\pi}$. In section 5, we observe
that its minimal injectivity radius is given by \hbox{$\nu(s)=\frac{\pi^2}{s}$} while the minimal injectivity
radius of \hbox{${\rm Dome}(\Omega(s))$} is given by \hbox{$\hat\nu(s)=\frac{\pi}{\sinh(s/2)}$}.
The minimal quasiconformal dilatation of a quasiconformal map from
$\Omega(s)$ to ${\rm Dome}(\Omega(s))$ is then given by
$$K(s)=\frac{\pi\sinh(s/2)}{s}\approx \frac{\nu(s)e^{\pi^2/2\nu(s)}}{2\pi}
\approx \frac{\pi^2}{2\hat\nu(s)\log(1/\hat\nu(s))}$$
(where the approximations hold as $s\to\infty$).
These examples indicate that our estimates ``almost'' have the correct asymptotic form.

In a final section, we obtain lower bounds on the quasiconformal constant when
a uniformly perfect domain has ``small'' injectivity radius. 

\medskip\noindent
{\bf Proposition \ref{lower bound}.} {\em
Let $\Omega$ be a uniformly perfect domain in $\rs$. Suppose that \hbox{$\phi:\Omega\to\chb$} is a $K$-quasiconformal
map which is homotopic to the nearest point retraction. If $\Omega$ contains
a point of injectivity radius $\nu\in(0,.5)$, in the Poincar\'e metric, then
$$K\ge {\nu e^{\pi^2\over 2\sqrt{e}\nu}\over \pi^2e^{\pi\over2}}=O\left(\nu e^{\pi^2\over 2\sqrt{e}\nu}\right).$$
}

\medskip\noindent
{\bf Acknowledgements:} The authors would like to thank the referees for their careful reading of
the paper and helpful suggestions.

\section{Historical Discussion}

The investigation of the relationship between the geometry of a domain and the boundary of the
convex core of its complement was initiated by Dennis Sullivan. We take the occasion of Dennis'
70th birthday as an excuse to provide a  partial 
tour of  results inspired by and/or related to Sullivan's Theorem.

\medskip\noindent
{\bf Sullivan's Theorem:} {\rm (\cite{sullivan})} 
{\em
There exists $K_0>1$ such that if $N=\Ht/\Gamma$ is a hyperbolic
3-manifold with finitely generated fundamental group and each component
of $\Omega(\Gamma)$ is simply connected, then the nearest point retraction
is homotopic to a \hbox{$K_0$-quasiconformal} map 
$$\phi:\partial_cN\to \partial C(N).$$
}

\medskip

Sullivan used his theorem in his exposition of Thurston's proof that 3-manifolds
fibering over the circle are geometrizable. 
(In fact, Sullivan only states his theorem in the setting of quasifuchsian hyperbolic
3-manifolds.) The key issue there
was to establish a relationship between
the geometry of the conformal boundary  and the internal geometry of the
hyperbolic 3-manifold.
From this viewpoint, Sullivan's Theorem may be
viewed as a vast generalization of a fundamental result of Bers.

\medskip\noindent
{\bf Bers' Theorem:} {\rm (\cite[Theorem 3]{bers-slice})} {\em  Suppose that
$N=\Ht/\Gamma$ is a hyperbolic
3-manifold with finitely generated fundamental group and each component
of $\Omega(\Gamma)$ is simply connected. If $\alpha$ is a closed geodesic in $\partial_cN$
(in the Poincar\'e metric) and $\alpha^*$ is the geodesic in $N$ in the homotopy
class of $\alpha$, then
$$l_N(\alpha^*)\le 2l_{\partial_cN}(\alpha)$$
where $l_N(\alpha^*)$ denotes the length of $\alpha^*$ in $N$ and
$l_{\partial_cN}(\alpha)$ denotes the length of $\alpha$ in $\partial_cN$.
}

\medskip

Bers also originally stated his result only in the setting of quasifuchsian
hyperbolic 3-manifolds. We note that in many applications, Bers' result provides
all the information one needs about the relationship between the geometry of
the conformal boundary and the geometry of the hyperbolic 3-manifold.

Another generalization of Bers' Theorem is provided by Epstein, Marden and Markovic:

\begin{theorem} 
{\rm (Epstein-Marden-Markovic \cite[Theorem 3.1]{EMM1})}
If $\Omega$ is a simply connected hyperbolic domain, then
the nearest point retraction is $2$-Lipschitz.
\end{theorem}

Bridgeman \cite{bridgeman} showed that, in the setting of Sullivans' Theorem, the nearest point retraction
\hbox{$\bar r:\partial_cN\to \partial C(N)$} has a \hbox{$(1+\frac{2\pi}{\log (3)})$}-Lipschitz
homotopy inverse. This paper also introduced the technique of bounding
the total bending along an arc in the boundary of the convex hull which was
used in many subsequent results.
Bishop \cite[Lemma 8]{bishop-divergence} observed that 
$r$ is a quasi-isometry with uniform quasi-isometry constants when $\Omega$
is simply connected. Bridgeman and Canary \cite{BC10}
provide explicit bounds

\begin{theorem}{\rm (Bridgeman-Canary \cite{BC10})}
If $\Omega$ is a simply connected hyperbolic domain, then
the nearest point retraction is a $(L',L_0')$-quasi-isometry, where
$L'\approx 4.56$ and $L_0'\approx 8.05$.
\end{theorem}

Another generalization of Bers' result is given by the length bounds on curves developed in
the solution of Thurston's Ending Lamination Conjecture, see \cite{ELC1,ELC2}.
These results also provide lower bounds on the length of closed geodesics in a quasifuchsian hyperbolic 3-manifold in terms of the geometry of the conformal boundary.

Epstein and Marden \cite{EM87} published the first complete proof of Sullivan's Theorem and 
noted that its proof can be extended to the setting of simply connected hyperbolic
domains in $\rs$. Moreover, they obtained the first explicit bounds on $K_0$.

\begin{theorem}{\rm (Epstein-Marden \cite{EM87})} If $\Omega$ is a simply
connected, hyperbolic domain in $\rs$, then there exists a  
conformally natural $K_0$-quasiconformal map
$$\phi:\Omega\to\chb$$
which admits a bounded homotopy to the nearest point retraction.
Moreover,
$$K_0\le 82.7.$$
\end{theorem}

Thurston \cite{thurstonII} asked
whether or not one could choose $K_0=2$. 
Epstein, Marden and Markovic \cite{EMM1,EMM2}  showed that 
that the optimal value of $K_0$ lies between $2.1$ and $13.88$
(see also Komori-Matthews \cite{komori-matthews} for an explicit  Kleinian group whose
domain of discontinuity provides a counterexample
to Thurston's $K=2$ question).
Bishop \cite{bishop} showed that if one does not require the quasiconformal map 
to be conformally natural, then one can bound the quasiconformal dilatation above by 7.88.
Epstein and Markovic \cite{EM05}, showed that even if one removes the requirement that the quasiconformal
map be conformally natural, one cannot always bound the quasiconformal dilatation above by 2.

The first generalization of Bers' result to the setting of hyperbolic manifolds
where the domain of discontinuity is not simply connected was provided by:

\begin{theorem}{\rm (Canary \cite{canary91})}
There exists a function $R: (0,\infty)\to (0,\infty)$ such that if
$N=\Ht/\Gamma$ is a hyperbolic 3-manifold,
$\nu>0$ is a lower bound for the injectivity radius of $\Omega(\Gamma)$, and
$\alpha$ is a closed geodesic in $\partial_cN$ of length $l(\alpha)$, then
$$l_N(\alpha^*) \le R(\nu) l (\alpha)$$
where $R(\nu)=\max \left\{ \sqrt{2}(k + log 2),\ \sqrt{2}\left( {k\nu + 8\pi k + 2\pi^2 \over \nu}\right) \right\}$
and $k=4+\log(3+2\sqrt{2})$.
\end{theorem}

Canary \cite{cbbc} later established bounds on lengths of curves in the boundary
of the convex core without assuming that domain of discontinuity was 
uniformly perfect.

\begin{theorem}{\rm (Canary \cite{cbbc})} Suppose that $N$ is a hyperbolic 3-manifold and that
\hbox{$r:\partial_cN\to \partial C(N)$} is the nearest point retraction
from its conformal boundary to the boundary of its convex core. If
$\alpha$ is a closed geodesic in the conformal boundary of length  $L$,
then
$$l_{\partial C(N)}(r(\alpha)^*)<\ 45 L e^{L\over 2}$$
where $l_{\partial C(N)}(r(\alpha)^*)$ denotes the length of
the closed geodesic $r(\alpha)^*$  in the homotopy class of $r(\alpha)$ in the intrinsic metric on $\partial C(N)$. \end{theorem}

Bridgeman and Canary \cite{BC03} showed that if $\nu>0$ is a lower bound for the injectivity radius
of $\Omega(\Gamma)$ and $\Gamma$ is finitely generated, then the nearest point retraction is
\hbox{$2\sqrt{2}(k+\frac{\pi^2}{2\nu})$}-Lipschitz. Moreover, the nearest point retraction
has a $P(\hat\nu)$-Lipschitz homotopy inverse where $P(\hat\nu)\approx{ 2\pi\over \hat\nu}$ as
$\hat\nu\to 0$.

Marden and Markovic \cite{MM} established the first direct analogue of Sullivan's Theorem
in the setting of uniformly perfect hyperbolic domains.

\begin{theorem}{\rm (Marden-Markovic \cite{MM})} 
Suppose that $\Omega$ is a hyperbolic domain in $\rs$. There exists a 
conformally natural quasiconformal
map $\phi:\Omega\to \chb$ which admits a bounded homotopy to
the nearest point retraction if and only if $\Omega$ is uniformly perfect.
Moreover, the quasiconformal dilatation of $\phi$ may be bounded by some
function of the injectivity radius of $\Omega$ in the Poincar\'e metric.
\end{theorem}

Marden and Markovic also showed that the nearest point retraction is a quasi-isometry if and only if
$\Omega$ is uniformly perfect. Again, the quasi-isometry constants depend on a lower bound  for the injectivity
radius. However, as their arguments make crucial use of compactness arguments,
Marden and Markovic do not find explicit bounds on the quasiconformal constants  or quasi-isometry
constants in terms of the injectivity radius.

In an earlier paper \cite{BC10}, we showed:

\begin{theorem}{\rm (Bridgeman-Canary \cite{BC10})}
\label{BCqi}
If  $\Omega$ is a uniformly perfect, hyperbolic domain in $\rs$ and $\nu>0$ is a lower bound on
its injectivity radius, then the nearest point retraction is \hbox{$2\sqrt{2}(k+\frac{\pi^2}{2\nu})$}-Lipschitz and 
is a \hbox{$((2\sqrt{2})(k+\frac{\pi^2}{2\nu}),L_0)$}-quasi-isometry,
where \hbox{$k=4+log(3+2\sqrt{2})$} and \hbox{$L_0\approx 7.12$}.
\end{theorem}

Theorem \ref{BCqi}  is obtained as  a corollary of a result which asserts that if
we give {\em any} hyperbolic domain $\Omega$ the Thurston metric, then
the nearest point retraction is 1-Lipschitz and a $(L,L_0)$-quasi-isometry where
$L\approx 8.49$ and $L_0\approx 7.12$. As another corollary of this result, we show
that the nearest point retraction is Lipschitz (with respect to the Poincar\'e metric on $\Omega$)
if and only if $\Omega$ is uniformly
perfect (which was originally conjectured by Marden and Markovic \cite{MM}.)

We also showed that  if $\Omega$ is uniformly perfect, $r$ lifts to a quasi-isometry between
the universal covers of $\Omega$ and $\chb$ (see Theorem \ref{r_lift_qi}).
We combined these results with work of Douady-Earle
\cite{douady-earle} to obtain 
explicit bounds, in terms of $\nu$, on the quasiconformal dilatation of the conformally
natural map between $\Omega$ and $\chb$.
However, our earlier bounds were very far from optimal, 
and the results in this paper provide a substantial improvement.

\section{Background}

In this section, we assemble the ingredients of our proof.
Section \ref{pleated} reviews the theory of pleated planes. Section \ref{npr}
recalls basic properties of the nearest point retraction and uses work of Bridgeman-Canary 
\cite{BC10} to give a bound on the ``roundness'' of the bending lamination of $\chb$.
Section \ref{ce} recalls work of Epstein, Marden, and Markovic \cite{EMM1} on complex earthquakes
which gives rise to a conformally natural quasiconformal map from $\Omega$ to $\chb$ when the roundness
of the bending lamination of $\chb$ is ``small.''
Section \ref{scaling} recalls the complex angle scaling maps used by Epstein, Marden, and Markovic \cite{EMM2}
to associate a family of quasiregular maps to a family of complex earthquakes.

\subsection{Pleated planes}
\label{pleated}

Throughout this paper we will use the disk model for $\Hp$ and identify
$\partial_\infty\Hp$ with $\Sp^1$.
Let $G(\Hp)$ be the set of unoriented geodesics on $\Hp$. 
We can identify
$G(\Hp) = (\Sp^1\times\Sp^1-\Delta)/\Z_2$.
A {\em geodesic lamination}  on $\Hp$ is a closed subset of $G(\Hp)$ such that  distinct geodesics in the subset
do not intersect.
A {\em measured lamination} $\mu$ on $\Hp$ is a non-negative measure on the space of geodesics $G(\Hp)$
supported on a geodesic lamination.
The measure $\mu$ assigns a measure to any arc transverse to the support by restricting $\mu$ to the set of geodesics that intersect the arc. 

In \cite{EMM1}, Epstein, Marden and Markovic defined the {\em roundness} of  a measured lamination $\mu$ to be
$$||\mu|| = \sup_C \mu(C)$$
where the sup is taken over all open geodesic arcs of unit length which are transverse
to $\mu$. 

In general, if $S=\Hp/\Gamma$ is a hyperbolic surface, then a measured lamination $\mu$
on $S$ lifts to a measured lamination $\tilde\mu$ on $\Hp$ and we define
the roundness
$$||\mu||=||\tilde\mu||.$$
Of course, $||\mu||$ can also be defined intrinsically.

Given a measured lamination $\mu$ on $\Hp$, one can define a convex pleated  plane 
$$P_\mu:\Hp \rightarrow \Ht$$
by fixing a component of the complement of $\mu$ and mapping all other components isometrically
after bending along the support of $\mu$ by an amount given by  the transverse measure on $\mu$
(see \cite[Chapter 3]{EM87}). The map $P_\mu$ is well-defined up to post-composition by an isometry.
In the case when $\mu$ is a finite collection of geodesics with discrete measure, $P_\mu$ is defined by mapping each
component of the complement of $\mu$ isometrically so that the the dihedral angle between any two
adjacent components is given by the value of the measure on that geodesic. 
For the general case, one can define $P_\mu$ by considering a sequence
$\{\mu_n\}$ of finite-leaved laminations converging to $\mu$ and letting
$P_\mu$ be the limit of $\{P_{\mu_n}\}$ (see \cite[Theorem 3.11.9]{EM87}).

Epstein, Marden and Markovic \cite{EMM1} showed that the pleated plane
is an embedding when the roundness is small.

\begin{theorem}{\rm (\cite[Theorem 4.2]{EMM1})}
\label{c2embed}
There exists a constant 
$$c_2 > .73$$ 
such that if $\mu$ is a measured lamination on $\Hp$ and
$||\mu|| \leq c_2$, then $P_\mu$ is a bilipschitz embedding. 
\end{theorem} 

\noindent
{\bf Remark:} Epstein and Jerrard (see \cite{BCpre1}) proved that one can take
$c_2\approx .948$ in Theorem \ref{c2embed}.

\subsection{The nearest point retraction and the bending lamination}
\label{npr}

Let $\Omega \subseteq \Sp^2$ be a hyperbolic domain and 
let $\chb$ be the boundary of the convex hull of $\rs-\Omega$.
The nearest point retraction
$$r:\Omega \rightarrow \chb$$ is a conformally natural homotopy equivalence
which extends continuously to a map from $\overline{\Omega}$ to
$\chb\cup\partial\Omega$ which is the identity on $\partial \Omega$.
In the special case that \hbox{$\rs-\Omega$} lies in a round
circle, \hbox{$CH(\rs-\Omega)$} is a subset of a totally geodesic plane in $\Ht$ and one
defines $\chb$ to be the double of \hbox{$CH(\rs-\Omega)$} along
its boundary (as a surface in the plane).
(See Epstein-Marden \cite[Chapter 1]{EM87} for an extensive discussion of the basic
properties of the nearest point retraction.)

Thurston \cite{ThBook}  observed that the intrinsic metric on $\chb$ is a complete hyperbolic metric.
Moreover, he showed that there exists a lamination $\mu$ on $\Hp$ such that
$\chb=P_\mu(\Hp)$ and that $P_\mu:\Hp\to\chb$ is a locally isometric covering
map.
Epstein and Marden \cite[Chapter 1]{EM87} developed
this observation more fully, see especially sections 1.11 and 1.12.

\begin{lemma}
\label{pleated plane and dome}
If $\Omega$ is a hyperbolic domain, there is a lamination $\mu$ on $\Hp$ such that
$P_{\mu}$ is a locally isometric covering
map with image $\chb$.
\end{lemma}

The lamination $\mu$ descends to a lamination $\hat\mu$ on $\chb$ which is called
the  {\em bending lamination} of $\chb$.

In \cite{BC03} the authors obtained a bound on the 
intersection number of a geodesic arc on $\chb$ with the bending lamination.
(We note that the result is stated in \cite{BC03} in the setting of domains of discontinuity of
Kleinian groups, but the proof goes through readily for general hyperbolic domains,
see also \cite{BC10}.)
We define \hbox{$F:[0,2\sinh^{-1}(1)) \rightarrow (0,\infty)$} and 
\hbox{$G:(0,\infty)\to [0,2\sinh^{-1}(1))$} by setting
$$F(x) = \frac{x}{2} + \sinh^{-1}\left(\frac{\sinh(x/2)}{\sqrt{1-\sinh^2(x/2)}}\right) \ \ \ 
{\rm and}\ \ \ G(x) = F^{-1}(x).$$

\begin{prop}{\rm (\cite[Proposition 5.1]{BC03})}
\label{bend bound}
Let $\Omega\subset\rs$ be a hyperbolic domain and let \hbox{$\hat\nu:\chb \rightarrow \R$} be the injectivity radius function on $\chb$.
Let $\mu$ be the bending lamination on $\chb$. 
If $\alpha:[0,1) \rightarrow \chb$ is a geodesic arc of length $l(\alpha) \leq G(\hat\nu(\alpha(0))$ which is
transverse to $\mu$, then
$$ \mu(\alpha) \leq 2\pi$$
\end{prop}

Proposition \ref{bend bound} may be re-interpreted to give an upper bound on the  roundness of
the bending lamination of $\chb$ when $\Omega$ is  uniformly perfect.

\begin{prop}{}
\label{norm bound}
Let $\Omega\subset\rs$ be uniformly perfect  hyperbolic domain and let  $\hat\nu>0$ be a lower bound
on the injectivity radius of $\chb$.
If $\mu$ is the bending lamination of $\chb$, then
$$||\mu|| \leq    2\pi \left\lceil{\frac{1}{G(\hat\nu)}}\right\rceil \le 
\frac{4\pi}{\hat{\nu}} +2\pi$$
where $\lceil x \rceil$ is the least integer greater than or equal to $x$.
 \end{prop}
 
\begin{proof} 
Since the function $G$ is increasing, 
Proposition \ref{bend bound}, implies that if $\alpha$ is a half open geodesic arc of length 
at most $G(\hat\nu)$ which is transverse to $\mu$, 
then $i(\mu,\alpha) \leq 2\pi$.  
Let \hbox{$\alpha:[0,1) \rightarrow \chb$} be a unit length geodesic arc  which is transverse to $\mu$.
We can decompose $\alpha$ into \hbox{$\lceil{1/G(\hat\nu)}\rceil$} intervals of length at most $G(\hat\nu)$.
By summing,  we see that
$$i(\alpha,\mu)\le 2\pi \left\lceil{\frac{1}{G(\hat\nu)}}\right\rceil.$$
Since $\alpha$  was arbitrary, we conclude that
$$||\mu||\le  2\pi \left\lceil{\frac{1}{G(\hat\nu)}}\right\rceil \le \frac{2\pi}{ G(\hat\nu)} +2\pi.$$

If $G(\hat \nu)\ge 1$, the final inequality follows immediately.
If \hbox{$G(\hat\nu)<1$,} then \hbox{$G(\hat\nu)\ge \hat\nu/2$},
and the final inequality follows from the fact that  
\hbox{$||\mu||\le   \frac{2\pi}{ G(\hat\nu)} +2\pi.$}
\end{proof}

In \cite{BC03} we showed that a lower bound on the injectivity radius of $\Omega$ implies an
explicit lower bound on the injectivity radius of $\chb$.

\begin{prop}
\label{inj radii}
{\rm (\cite[Lemma 8.1]{BC03})}
If $\Omega\subset\rs$ is a uniformly perfect domain and $\nu$ is a lower bound for its injectivity radius,
then $g(\nu)$ is a lower bound for the injectivity radius of $\chb$, where
$$g(\nu) = \frac{e^{-m}e^{-\frac{\pi^2}{2\nu}}}{2}\ \ \  {\rm and}\ \ \  m = \cosh^{-1}(e^2).$$ 
\end{prop}

Combining Propositions \ref{norm bound} and \ref{inj radii} we immediately obtain:

 \begin{corollary}
 \label{norm bound 2}
Let $\Omega$ be a uniformly perfect  hyperbolic domain and let  $\nu>0$ be a lower bound on the injectivity radius 
of $\Omega$.
If $\mu$ is the bending lamination of $\chb$, then
$$||\mu||  \leq 8\pi e^m e^{\frac{\pi^2}{2\nu}} +2\pi
\le 370 e^{\frac{\pi^2}{2\nu}} +2\pi.$$
 \end{corollary}

In \cite{BC10} the authors established that, if $\Omega$ is uniformly perfect, the nearest point retraction lifts
to a quasi-isometry
between the universal cover $\tilde \Omega$ of $\Omega$ and the
universal cover $\chbu$ of  $\chb$. 

\begin{theorem}{\rm (\cite[Theorem 1.5]{BC10})}
Suppose that $\Omega$ is a uniformly perfect hyperbolic domain and
$\nu>0$ is a lower bound for its injectivity radius in the Poincar\'e metric.
Then the nearest point retraction lifts to a quasi-isometry
$$\tilde r:\tilde\Omega \to \chbu$$
between the universal cover of $\Omega$ and the universal cover
of $\chb$ where the quasi-isometry constants depend only on $\nu$.
\label{r_lift_qi}
\end{theorem}

\subsection{Complex earthquakes}
\label{ce}

In this section, we recall the definition of a complex earthquake associated to a lamination (originally
introduced by Epstein-Marden \cite{EM87}.)  We then describe work of Epstein, Marden and Markovic \cite{EMM2} 
which uses the theory of
holomorphic motions  to  associate a quasiconformal map to  a complex earthquake along a lamination with ``small roundness.'' As a consequence, one obtains a conformally natural 
quasiconformal map  from
$\Omega$ to $\chb$ when the bending lamination of $\chb$ has ``small roundness.''

If $\mu$ is a measured lamination on $\Hp$,  one defines the earthquake map $E_{\mu}: \Hp \rightarrow \Hp$ 
by fixing a component of the complement of $\mu$ and (left-)shearing all other components by an amount given by the measure on $\mu$
(see Thurston \cite{thurston-earthquakes} or Epstein-Marden \cite{EM87} for careful discussions).
An earthquake map $E_\mu$ is continuous  except on leaves of $\mu$ with discrete measure 
(see \cite[Section II.3.6]{EMM2}) and extends to a homeomorphism of $S^1=\partial\Hp$
(see \cite[Proposition III.1.2.6]{thurston-earthquakes}). It follows that
a measured lamination $\lambda$ on $\Hp$ is mapped to a well-defined measured lamination
$E_{\mu}(\lambda)$. 

Given a measured lamination $\mu$ and $z=x+iy\in \C$, we  define the {\em complex earthquake}
$$\C E_z: \Hp \rightarrow \Ht$$ to be the result of first 
earthquaking along $x\mu$ and then bending along $yE_{x\mu}(\mu)$
(see Epstein-Marden \cite{EM87} for a complete discussion and Epstein-Marden-Markovic
\cite[Section 4]{EMM1} for a discussion in this notation).
Formally, 
$$\C E_{z} = P_{yE_{x\mu}(\mu)}\circ E_{x\mu}.$$

In \cite{EMM1}, Epstein, Marden and Markovic  find an open neighborhood $\mathcal{T}_0$ of the origin such
that if $t\in\mathcal{T}_0$,  the complex earthquake map 
$\C E_t$ extends  continuously to $\partial_\infty \Hp = \Sp^1$ to define a map $\phi_t:\Sp^1\to\rs$,
such that $\phi_t(\Sp^1)$ is a Jordan curve bounding a domain $\Omega_t$. They also produce 
a quasiconformal map $\Phi_t:\Hp\to\Omega_t$.

Let
$$ f(L,x) = \min(\sinh^{-1}(e^{|x|}\sinh{L}), e^{|x|/2}\sinh{L})$$
and define a domain ${\mathcal T}_0\subseteq \C$ by
$${\mathcal T}_0 = \left\{ x+iy \ \left| \ |y| < \frac{c_2}{{\lceil f(1,x)\rceil}}\right\}\right.  .$$

\begin{theorem}{\rm (\cite[Theorem 4.14]{EMM1})}
Let $\mu$ be a measured lamination on $\Hp$ such that $||\mu|| = 1$.
Then, for $t \in {\mathcal T}_0$, we have
\begin{enumerate}
\item
$\C E_t$ extends to an embedding $\phi_t:\Sp^1\to \rs$ which bounds a region $\Omega_t$. 
\item
If $t=x+iy$ and $y>0$, the bending measure of $\chbt$ is $yE_{x\mu}(\mu) $.
\item 
There is a  quasiconformal map
$\Phi_t :\Hp \rightarrow \Omega_t$  with dilatation $K_t$ given  by 
$$K_t \leq \frac{1 + |h(t)|}{1-|h(t)|}$$
where $h:{\mathcal T}_0 \rightarrow \Hp$ is a Riemann map taking $0$ to $0$.

In particular, $K_t\le 2$ if $|t|<{1\over 3}$.
\item 
$\Phi_t\cup\phi_t:\Hp\cup \Sp^1\to \rs$ is continuous.
\item If $G$ is a group of M\"obius transformations preserving $\mu$, then $\Phi_t$ can be chosen so that there is a homomorphism $\rho_t:G \to G_t$ where 
$G_t$ is also a group of M\"obius transformations and 
$$\Phi_t \circ g = \rho_t(g)\circ \Phi_t$$
for all $g\in G$.

\end{enumerate}
\label{holo motion}
\end{theorem}

One obtains the following immediate corollary, whose proof we sketch as a warm-up
for the proof of our main result.

\begin{corollary}{\rm (Epstein-Marden-Markovic \cite{EMM2})}
\label{smallsullivan}
If $\Omega\subset \rs$ is a simply connected, hyperbolic domain, $\mu$
is the bending lamination of $\chb$ and \hbox{$||\mu||< c_2$}, then there is a conformally
natural $K_t$-quasiconformal map \hbox{$\phi:\Omega\to \chb$} which admits a bounded homotopy
to the nearest point retraction.
\end{corollary}

\begin{proof}
Let $y=||\mu||$ and $t=iy$. Since $y\in (0,c_2)$, $t\in\mathcal{T}_0$.
Let 
$$\Phi_t:\Hp\to \Omega_t$$ 
be the $K_t$-quasiconformal map provided by Theorem \ref{holo motion}.
Theorem \ref{holo motion}(2) allows us to conclude that $\Omega_t=\Omega$,
so
$$\phi=P_\mu\circ \Phi_t^{-1}:\Omega\to\chb$$
is a $K_t$-quasiconformal map.

Since $\phi$ is quasiconformal, it is a quasi-isometry (see, e.g, \cite[Theorem 4.3.2]{FM07}).
Since $\Omega$ is a Jordan domain and the nearest point retraction $r$ and $\phi$ are quasi-isometries  which
agree on $\partial\Omega$, there is a bounded homotopy from $\phi$ to $r$ (see, e.g.
\cite[Proposition 4.3.1]{FM07}). 

Let $\Gamma$ be a group of M\"obius transformations
preserving $\Omega$. Then it preserves $\chb$, and hence its bending lamination.
We may pull-back the restriction of the action of $\Gamma$ on $\chb$, via $P_\mu$, 
to obtain a group $\Gamma'$
of conformal automorphisms of $\Hp$ which preserve $\mu$. Theorem \ref{holo motion}(5)
implies that we may choose $\Phi_t$ so that there exists a homomorphism 
\hbox{$\rho_t:\Gamma'\to\Gamma_t$} so that $\Phi_t$ is $\rho_t$-equivariant.
One may readily check that $\Gamma_t$ must agree with $\Gamma$, since they agree on $\partial\Omega$, and hence
that $\phi$ is conformally natural. 
\end{proof}

\subsection{Complex angle scaling}
\label{scaling}

Let $\mu$ be a measured lamination on $\Hp$ with unit roundness, i.e. $||\mu||=1$,
and let $t_0=iy_0$ be a point in $\U\cap  {\mathcal T}_0$
(where $\U$ is the upper half plane). Let $\Omega_0=\Omega_{t_0}$ be
the domain produced by Theorem \ref{holo motion}. In \cite{EMM2}
Epstein, Marden and Markovic  define a map
\hbox{$\Psi:\U \times \Omega_{0} \rightarrow \rs$} such that
each \hbox{$\Psi_t=\Psi(t,\cdot):\Omega_0\to \rs$} is  a locally injective quasiregular map.
We recall that a locally injective map $f:\Omega_0\to\rs$ is \hbox{$K$-{\em quasiregular}} if $f = h\circ g$ where
$g$ is a \hbox{$K$-quasiconformal} homeomorphism and 
$h$ is locally injective and holomorphic on the image of $g$. 
In this section, we describe their construction and its relevant properties. 

A key tool in their construction is the complex angle scaling map defined on crescents.
A {\em crescent} $C$ is a region bounded by two circles which intersect transversely.
Any crescent with interior angle $\theta$ can be mapped by a M\"obius transformation to
$$W_\theta= \left\{ z \ \left| \  0 \leq \arg(z) \leq \theta\right.\right\}.$$ 
Given $w\in\C$, we
define an angle-scaling map on $W_\theta$ by setting 
$$S_w(z) = ze^{w\arg(z)\theta}.$$
If \hbox{$({\rm Im}(w)+1)\theta < 2\pi$}, then $S_w$ is a quasiconformal homeomorphism
onto a crescent with interior angle \hbox{$({\rm Im}(w)+1)\theta$}.
In general, a {\em complex angle scaling} map with domain a crescent $C$ with interior
angle $\theta$ has the form
$\gamma S_w\beta $ for some $w$ where $\gamma$ and $\beta$ are M\"obius transformations
and $\beta(C)=W_\theta$.

Let $\mathcal{N}\subset T_1(\Ht)$ denote the set of unit tangent vectors
to geodesic rays $\overrightarrow{r(z)z}$ at the point $r(z)$ where $z\in\Omega_0$.
Informally, $\mathcal{N}$ is the set of normal vectors to $\chbz$. 
Let $\nu:\Omega_0\to \mathcal{N}$ be the obvious homeomorphism.
If we let ${\rm exp}_+:T_1(\Ht)\to \rs$ be the exponential map which takes a vector
to the end point of its corresponding infinite geodesic ray, then
${\rm exp}_+|_{\mathcal N}:\mathcal{N}\to \Omega_0$ is the inverse to
$\nu$ (see \cite[Lemma 4.9]{EMM2}).

For $t\in \C$, they define a map $N_t:\mathcal{N}\to T_1(\Ht)$ with the property
that if $p$ is the basepoint of a vector $n\in\mathcal{N}$ and $p$ lies in a gap or on a
bending line which does not have an atomic weight, then $N_t(n)$ is
the unique normal vector to the image of $\C E_t$ which has the correct orientation.
They then define a map
$\Psi_t:\Omega_0\to \rs$ by setting 
$\Psi_t(z)={\rm exp}_+\circ N_t\circ \nu$ if $r_0(z)$ 
lies in a gap or on a
bending line which does not have an atomic weight. Otherwise,
$r_0(z)$ lies in a crescent $r_0^{-1}(l)$ in $\Omega_0$ which maps to a bending line $l$ with positive measure.
On $r_0^{-1}(l)$, $\Psi_t$ is conjugate to the complex angle scaling map $S_{i(t-t_0)/t_0}$.
In particular, if $\mu(P_{t_0\mu}^{-1}(l))=\theta$, $t\in (0,\infty)$ and $t\theta < 2\pi$, then $\Psi_t$ 
is quasiconformal on
$r_0^{-1}(l)$ and its
image is a crescent of total angle $t\theta$ with ``endpoints'' $\partial \C E_t(l)$.
The map $\Psi:\U \times \Omega_{t_0} \rightarrow \rs$ is then defined by
setting $\Psi(t,z) = \Psi_t(z)$.

We summarize their results below:

\begin{theorem}{\rm (Epstein-Marden-Markovic \cite[Theorem 4.13]{EMM2})}
Let $\mu$ be a measured lamination on $\Hp$ such that $||\mu|| = 1$ and 
let $t_0 = iy_0 \in\U\cap {\mathcal T}_0$. Then there exists a continuous map
$\Psi:\U \times \Omega_{0} \rightarrow \rs$ such that
\begin{enumerate}
\item 
The map $\Psi_{t_0} = \mbox{id}$.
\item 
For $t \in {\mathcal T}_0$, $\Psi_t$ can be extended continuously to
a map \hbox{$\psi_t:\partial \Omega_{0}\to\rs$} so that \hbox{$\psi_t \circ \phi_{t_0} = \phi_t$}.
\item
For $t \in \U\cap{\mathcal T}_0$, $\Psi_t$ is injective and 
$\Psi_t(\Omega_0) = \Phi_t(\Hp) = \Omega_t$.
\item 
The map $\Psi_t$ is a locally injective  $L_t$-quasiregular mapping, where 
$$L_t = \frac{1 + |\kappa(t)|}{1-|\kappa(t)|}\ \ \ {\rm and}\ \ \  |\kappa(t)| =
\left| \frac{t-t_0}{t+t_0}\right|.$$
\item
If $G$ is a group of M\"obius transformations preserving $\Omega_0$,
then there is a homomorphism $\rho_t:G\to  G_t$ where 
$G_t$ is also a group of M\"obius transformations, such that 
$$\Psi_t \circ g = \rho_t(g)\circ \Psi_t$$
for all $g\in G$.
\end{enumerate}
\label{scale}
\end{theorem}

{\bf Remark:} The equivariance property (5) is not included in the statement of Theorem 4.13
in \cite{EMM2}, but is discussed before the proof of Theorem 6.13 in \cite{EMM2}.
It follows readily from  equivariance properties of complex earthquakes and the
explicit description of the map $\Psi$ given above.

\section{Bounds on quasiconformal dilatation}

We are now ready to prove our main results. 
We first provide a sketch of the argument. Given a uniformly perfect domain $\Omega$, 
there exists a lamination $\mu$ so that \hbox{$P_\mu:\Hp\to\chb$} is  a universal covering map.
Let $\mu_0=\mu/||\mu||$. If we choose $t_0=iy_0\in\mathcal{T}_0$, then
Theorem \ref{holo motion} provides a $K_{t_0}$-quasiconformal map  
$\Phi_0:\Hp\to\Omega_0=\Omega_{t_0}$.
Let $\Psi$ be the map associated to $\mu_0$ and $t_0$ provided by Theorem \ref{scale}
and let $k=||\mu||i$.
We use the explicit description of $\Psi$ to check that $\Psi_k$ is a
$L_k$-quasiregular covering map with image
$\Omega$. We check that $\Psi_k\circ\Phi_0$ descends,
under the covering map $P_\mu$, to a
\hbox{$K_{t_0}L_k$}-quasiconformal map \hbox{$f:\chb\to\Omega$}. We then check that $\phi=f^{-1}$ has the desired properties.

\medskip\noindent
{\bf Proof of Theorems \ref{main} and \ref{main2}:} 
Let $\hat\mu$ be the bending lamination of $\chb$. Then $\hat \mu$ lifts to give a measured lamination
$\mu$ on $\Hp$. One may normalize so that $P_\mu:\Hp\to\chb$ is a locally isometric, universal
covering map (see Lemma \ref{pleated plane and dome}).
Let $G$ be the group of covering transformations of $P_\mu$.

Let $k=||\mu || i$ and define $\mu_0={1\over |k|}\mu$,
so $\mu = |k| \mu_0$ and $||\mu_0|| = 1$.

Choose $t_0=iy_0\in \U\cap{\mathcal T}_0$. Theorem \ref{holo motion} gives a $K_{t_0}$-quasiconformal map 
$$\Phi_{t_0}:\Hp\rightarrow \Omega_{t_0}$$
whose image is a Jordan
domain $\Omega_{t_0}$. For simplicity, we denote
$\Phi_0=\Phi_{t_0}$ and $\Omega_0=\Omega_{t_0}$.
Let $\rho_0=\rho_{t_0}$ be the homomorphism provided by Theorem \ref{holo motion} and
let $G_0=\rho_0(G)$.
Let $r_0:\Omega_0 \rightarrow \chbz$
be the nearest point retraction.
Notice that 
$$P_0=P_{t_0\mu_0}:\Hp\to\chbz$$
is an isometry which
extends continuously  to a map $\phi_0:\Sp^1\to \rs$ 
and that
$$\Phi_0\cup\phi_0:\Hp\cup \Sp^1\to \rs$$
is continuous
(see Lemma \ref{pleated plane and dome} and Theorem \ref{holo motion}).
Then 
$$F_0=P_0\circ \Phi_0^{-1}:\Omega_0\to\chbz$$
is a conformally natural $K_{t_0}$-quasiconformal homeomorphism which
extends to the identity on $\partial\Omega_0$
(see Corollary \ref{smallsullivan}).

Let $\Psi:\U\times\Omega_0\rightarrow \rs$ be the map associated to $\mu_0$ and $t_0$ provided by Theorem \ref{scale}.
Consider the  map \hbox{$\Psi_k:\Omega_0 \rightarrow \rs$}. Notice that $\C E_k=P_\mu$.
So, by the definition of $\Psi$, if  $p$ lies in a component of 
$\Hp-\mu$ or on a bending line which does not have atomic weight, then
$\Psi_k(r_0^{-1}(P_0(p))=r^{-1}(P_\mu(p))$. Moreover, if $l$ is a leaf with atomic weight, 
then  $\Psi_k$ restricts to a  quasiconformal homeomorphism from $r_0^{-1}(P_0(l))$ to $r^{-1}(P_\mu(l))$ given by a complex angle scaling 
map between the two crescents. 

Therefore, $\Psi_k$ is a conformally natural, locally injective \hbox{$L_k$-quasiregular} map with image $\Omega$.
Moreover, $\Psi_k$ is a topological universal covering map with covering transformations 
$G_0$.

The composition 
$$F=\Psi_k\circ \Phi_0:\Hp\to \Omega$$
is a $K_{t_0}L_k$-quasiregular universal covering map
with covering transformations $G$. Therefore, it descends, under the covering map
$P_\mu:\Hp\to \chb$, to a $K_{t_0}L_k$-quasiconformal homeomorphism
$$f:\chb\to \Omega.$$

\hspace{1.7in}\xymatrix{
&\Hp \ar[ld]_{P_\mu}\ar@{-->}[ldd]^{F}\ar@{-->}[rdd]_{\Phi_0} \ar[rd]^{P_0}\\
\mbox{Dome}(\Omega) &  & \mbox{Dome}(\Omega_0)\\
\Omega \ar[u]^r & & \Omega_0 \ar[u]_{r_0}\ar[ll]^{\Psi_k} }

We claim that $f$ is a homotopy inverse for $r$ and that there is a bounded homotopy between
$\phi=f^{-1}$ and $r$. Consider the conformal universal covering map 
\hbox{$p:\tilde\Omega\to\Omega$ of $\Omega$.}
Let $\tilde G$ denote the group of covering transformations of the universal covering map $p$.
The quasiregular map $F$ lifts to a quasiconformal map $\tilde F:\Hp\to\tilde\Omega$ which conjugates the
action of $G$ to the action of $\tilde G$.
The maps $\tilde F$ and $\Phi_0$ are quasi-isometries,
since they are quasiconformal (see, e.g. \cite[Theorem 4.3.2]{FM07}).

Let $l$ be a geodesic in $\Hp$ which is either disjoint from $\mu$ or is a leaf of $\mu$. 
Let $a$ and $b$ be the endpoints of $l$.  Then
$\Phi_0(l)$ is a quasigeodesic with endpoints $\phi_0(a)$ and $\phi_0(b)$.
Let $A_l=r_0^{-1}(P_0(l))$.  Since there is a bounded homotopy between $F_0$ and $r_0$,
there is a bounded Hausdorff distance between $\Phi_0(l)$ and $A_l$. In particular, $\Phi_0(l)$ and
$A_l$ accumulate at the same endpoints at infinity. By construction 
$$\Psi_k(A_l)=r^{-1}(P_k(l))=B_l$$ 
and $\Psi_k$ is a complex angle scaling map on $A_l$.
Therefore, $F(l)$ is a quasigeodesic which lies a bounded Hausdorff distance from $B_l$.

Recall that $P_\mu:\Hp\to \chb$ is a locally isometric covering map, so this gives an explicit identification
of the universal cover $\chbu$ of $\chb$ with $\Hp$. With this identification, 
\hbox{$\tilde F:\Hp\to\tilde\Omega$}
is a lift of \hbox{$f:\chb\to \Omega$.} Fix a leaf $l_0$ of $\mu$. We may then choose a lift 
\hbox{$\tilde r:\tilde\Omega\to\chbu=\Hp$}
so that \hbox{$\tilde r(\tilde F(l_0))$} is a quasigeodesic which is a  bounded Hausdorff distance from $l_0$.
With this normalization,  if $l$ is any geodesic which either lies in $\mu$ or is disjoint from $\mu$,
then \hbox{$\tilde r(\tilde F(l))$} is a quasigeodesic which is a  bounded Hausdorff distance from $l$.
Since every point in $\Sp^1$ is an endpoint of such a geodesic and \hbox{$\tilde r\circ\tilde F$} is a quasi-isometry
(being a composition of quasi-isometries), we see that \hbox{$\tilde r\circ \tilde F$} extends to the identity
on $\Sp^1$ and the straight-line homotopy between \hbox{$\tilde r\circ \tilde F$} and the identity map on $\Hp$
is a bounded homotopy (see, e.g. \cite[Proposition 4.3.1]{FM07}).
It follows, since $\tilde F$ is a quasi-isometry and a homeomorphism, that the straight-line homotopy between $\tilde r$ and $\tilde F^{-1}$ is also a
bounded homotopy.
This homotopy descends to give a bounded homotopy between $r$ and
$\phi=f^{-1}$.

Let $\Gamma$ be the group of conformal automorphisms of $\Omega$. $\Gamma$ is also
a group of conformal automorphisms of $\chb$ which preserves the bending lamination.
Let $\hat\Gamma$ be the group of all conformal automorphisms of $\Hp$ which
arise as lifts, under the covering map $P_\mu$, of restrictions of elements of $\Gamma$  to 
$\chb$. Notice that
$G\subset\hat\Gamma$ and that $\hat\Gamma$ preserves $\mu$. 
Theorems \ref{holo motion} and \ref{scale} insure that we may assume that
there exist homomorphisms $\rho_0:\hat \Gamma\to \hat\Gamma_0$ and
$\rho_k:\hat\Gamma_0\to \hat\Gamma_k$ so that $\Phi_0$ is $\rho_0$-equivariant
and $\Psi_k$ is $\rho_k$-equivariant. Notice that $\rho_0(G)=G_0$,
$\rho_k(G_0)$ is trivial, $\rho_k(\hat\Gamma_0)=\Gamma$, and
\hbox{$F:\Hp\to \Omega$} is \hbox{$(\rho_k\circ\rho_0)$}-equivariant. Since the action of $\hat\Gamma$ on
$\Hp$ descends, via $P_\mu$, to the action of $\Gamma$ on $\chb$, it follows that
$f$, and hence $\phi$, is conformally natural.

Therefore, $\phi$ is a conformally natural $K_{t_0}L_k$-quasiconformal map from $\Omega$
to $\chb$ which admits a bounded homotopy to the nearest point retraction.
It only remains to check the claimed bounds on the quasiconformal dilatation.

Theorem \ref{scale} implies that
$$L_k={||\mu||\over y_0},$$
so
$$K_{t_0}L_k= \frac{K_{t_0}}{y_0} ||\mu||$$ 

We can choose $y_0\in (0,1/3)$ arbitrarily close to $1/3$ in which case we can choose $K_{t_0}=2$,
so we may assume that
$$K_{t_0}L_k\le 6||\mu||.$$

Applying Proposition \ref{norm bound} 
and Corollary \ref{norm bound 2}, which give  bounds on $||\mu||$ in terms of  the injectivity radii of
$\Omega$ and $\chb$, we obtain   
$$ K_{t_0}L_k\le 6\left(\frac{4\pi}{\hat{\nu}} + 2\pi\right) = N(\hat\nu)$$
and
$$K_{t_0}L_k\le 6 \left(8\pi e^m e^{\frac{\pi^2}{2\nu}} + 2\pi\right)  = M(\nu).$$
It follows that $\phi$ is $M(\nu)$-quasiconformal and $N(\hat\nu)$-quasiconformal as claimed,
which completes the proof.
\eproof

\section{Round annuli}

Let $\Omega(s)$ denote the round annulus lying between concentric circles of radius $1$ and $e^{ s}>1$ about the origin. 
Then $\Omega(s)$ has conformal modulus $s/2\pi$. (Recall that if $A\subset\rs$ is a hyperbolic domain which is homeomorphic
to an annulus, then there exists a unique $s_A>1$ such that $A$ is conformal to $\Omega(s_A)$ and one
defines its conformal modulus to be \hbox{${\rm mod}(A)=s_A/2\pi$}.)
In the Poincar\'e metric  $\Omega(s)$ is a complete hyperbolic annulus with core
curve of length 
$${2\pi^2\over s}={\pi\over {\rm mod}(\Omega(s))}.$$
In its intrinsic metric,
${\rm Dome}(\Omega(s))$ is a complete hyperbolic annulus with core curve of
length $2\pi\over \sinh({s\over 2})\ $.
In particular, ${\rm Dome}(\Omega(s))$ has modulus
$${\rm mod}({\rm Dome}(\Omega(s)))=\sinh(s/2)/2.$$
Therefore, the minimal quasiconformal dilatation $K(s)$ of a quasiconformal map 
between $\Omega(s)$ and  ${\rm Dome}(\Omega(s))$ is given by
$$K(s)={\pi\sinh(s/2)\over s}.$$

The injectivity radius bounds are given by
$$\nu(s)=\nu(\Omega(s))={\pi^2\over s}\ \ \ {\rm and}\ \ \ \hat \nu(s)=\hat\nu(\Omega(s))={\pi\over\sinh(s/2)}.$$
Therefore, as $s\to\infty$ and $\nu\to0$,
$$K(s) = \frac{\nu(s)\sinh(\frac{\pi^2}{2\nu(s)})}\pi \approx \frac{\nu(s)  e^{\frac{\pi^2}{2\nu(s)}}}{2\pi} = O\left(\frac{M(\nu(s))}{\log(M(\nu(s)))}\right)$$
and $$K(s) = \frac{\pi^2}{2\hat\nu(s)\sinh^{-1}(\pi/\hat\nu(s))} \approx \frac{\pi^2}{2\hat\nu(s)\log(1/\hat\nu(s))} = O\left(\frac{N(\hat\nu(s))}{\log(N(\hat\nu(s)))}\right).$$

\section{General lower bounds on the quasiconformal constant}

In this section, we give an explicit lower bounds on the dilatation of a quasiconformal map from $\Omega$
to $\chb$ which admits a bounded homotopy to the nearest point retraction. As in the last section, our lower bound
has ``almost'' the same asymptotics as the upper bound in our main theorem.

\begin{prop}
\label{lower bound}
Let $\Omega$ be a uniformly perfect domain in $\rs$. Suppose that \hbox{$\phi:\Omega\to\chb$} is a $K$-quasiconformal
map which is homotopic to the nearest point retraction. If $\Omega$ contains
a point of injectivity radius $\nu\in(0,.5)$, in the Poincar\'e metric, then
$$K\ge {\nu e^{\pi^2\over 2 \sqrt{e}\nu}\over \pi^2 e^{\pi\over2}}=O\left(\nu e^{\pi^2\over 2\sqrt{e}\nu}\right).$$
\end{prop}

\begin{proof}
Let $\gamma$ be a simple closed geodesic in $\Omega$ with length $L\le2\nu$. Theorem 5.1 in \cite{cbbc} implies
that $r(\gamma)$ is homotopic to a closed geodesic $r(\gamma)^*$ in $\chb$ with length
$$L'<{4\pi e^{.502\pi}\over e^{\pi^2\over \sqrt{e}L}}<.153.$$

Maskit \cite[Proposition 1]{maskit-extremal}  and Sugawa \cite[Theorem 5.2]{sugawa}
showed that if  $\alpha$ is a simple closed geodesic  of length $l$ 
in a hyperbolic surface $X$ and
$M(\alpha)$ is the maximal modulus of an
annulus in $X$ whose core curve is homotopic to $\alpha$, then
$$ {\pi\over l}\ge M(\alpha)\ge{\pi\over l e^{l/2}}.$$

Therefore, $r(\gamma)^*$ is homotopic to the core curve of  an annulus $A'$  in $\chb$ with modulus
$${\rm mod}(A')\ge {\pi \over L'e^{L'\over 2}} > {2\over L'}.$$
Since $\phi$ is $K$-quasiconformal, $A=\phi^{-1}(A')$  has modulus
$${\rm mod}(A)\ge {{\rm mod}(A')\over K}\ge  {2\over KL'}.$$
On the other hand, since the core curve of $A$ is homotopic to $\gamma$, 
$${\rm mod}(A)\le {\pi\over L},$$
so
$$K\ge {2L\over \pi L'}\ge  {Le^{\pi^2\over\sqrt{e}L}\over 2\pi^2 e^{.502\pi}}\ge {\nu e^{\pi^2\over 2\sqrt{e}\nu}\over \pi^2e^{\pi\over 2}}$$
as claimed.
(In the final inequality we use the fact that \hbox{$h(\nu)=\nu  e^{\pi^2\over 2\sqrt{e}\nu}$} is decreasing
on the interval  $(0,.5)$.)
\end{proof}


\begin{thebibliography}{99}
 
{\footnotesize
 
\bibitem{bers-slice} L. Bers, ``On boundaries of Teichm\"uller spaces and
Kleinian groups I,'' {\em Annals of Math.} {\bf 91}(1970), 570--600.

\bibitem{bishop-divergence} C.J. Bishop, ``Divergence groups have the Bowen property,''
{\em Annals of Math.} {\bf 154}(2001), 205--217.

\bibitem{bishop} C.J. Bishop, ``An explicit constant for Sullivan's convex hull
theorem,'' in {\em In the Tradition of Ahlfors and Bers III}, {\em Contemp. Math.}
{\bf 355}(2007), Amer. Math. Soc., 41--69.
  
\bibitem{bridgeman} M. Bridgeman, ``Average bending of convex pleated planes in 
hyperbolic three-space,'' {\em Invent. Math.} {\bf 132}(1998), 381--391.
  
\bibitem{BC03} M. Bridgeman and R.D. Canary,  ``From the boundary of the convex core to the conformal boundary,'' {\em Geom. Ded.} {\bf 96}(2003), 211--240.

\bibitem{BC10} M. Bridgeman and R.D. Canary,  ``The Thurston metric on hyperbolic domains and boundaries of convex hulls,'' {\em G.A.F.A.} {\bf 20}(2010), 1317--1353.

\bibitem{BCpre1} M. Bridgeman, R.D. Canary and A. Yarmola, in preparation.

\bibitem{ELC2} J. Brock, R.D. Canary and Y. Minsky,
``The classification of Kleinian
surface groups II: the ending lamination conjecture,'' 
{\em Annals of Math.},  {\bf 176}(2012), 1--149.

\bibitem{canary91} R.D. Canary, ``The Poincar\'e metric and a conformal version of a theorem of Thurston,''
{\em Duke Math. J.} {\bf 64}(1991), 349--359.

\bibitem{cbbc} R.D. Canary, ``The conformal boundary and the boundary
of the convex core,'' {\em Duke Math. J.} {\bf 106}(2001), 193--207.

\bibitem{douady-earle} A. Douady and C. Earle, ``Conformally natural
extensions of homeomorphisms of the circle,'' {\em Acta. Math.}
{\bf 157}(1986), 23--48.


\bibitem{EM87} D.B.A. Epstein and A. Marden, ``Convex hulls in hyperbolic
space, a theorem of Sullivan, and measured pleated surfaces,'' in
{\em Analytical and Geometrical Aspects of Hyperbolic Space}, Cambridge
University Press, 1987, 113--253.

\bibitem{EMM1} D.B.A. Epstein, A. Marden and V. Markovic, 
``Quasiconformal homeomorphisms and the convex hull boundary,''
{\em Annals of Math.} {\bf 159}(2004), 305--336.
 
\bibitem{EMM2} D.B.A. Epstein, A. Marden and V. Markovic, 
``Complex earthquakes and deformations of the unit disk,''
{\em J. Diff. Geom} {\bf 73}(2006), 119--166.

\bibitem{EM05} D.B.A. Epstein and V. Markovic, ``The logarithmic spiral: a counterexample to the $K=2$ conjecture,'' {\em Annals of Math.} {\bf 161}(2005), 925--957.

\bibitem{FM07} A. Fletcher and V. Markovic, {\em Quasiconformal Maps and Teichm\"uller Theory}, Oxford Graduate Texts in Mathematics, 2007.

\bibitem{komori-matthews} Y. Komori and C. Matthews, ``An explicit counterexample to
the equivariant $K=2$ conjecture.'' {\em Conform. Geom. Dyn.} {\bf 10}(2006), 184--196.
 
\bibitem{MM} A. Marden and V. Markovic, ``Characterizations of plane regions
that support quasiconformal mappings to their domes,'' {\em Bull. L.M.S.}
{\bf 39}(2007), 962--972.
 
\bibitem{maskit-extremal} B. Maskit, ``Comparison of hyperbolic and extremal lengths,'' 
{\em Ann. Acad. Sci. Fenn.} {\bf 10}(1985), 381--386.

\bibitem{ELC1}
Y. Minsky, ``The classification of {Kleinian} surface groups {I}: models and
bounds,'' {\em  Annals of Math.}, {\bf 171}(2010), 1--107.

\bibitem{pommerenke} C. Pommerenke, ``On uniformly perfect sets and Fuchsian
groups,'' {\em Analysis} {\bf 4}(1984) 299--321.

\bibitem{sugawa} T.Sugawa, ``Various domain constants related to uniform perfectness,''
{\em Comp. Var. Th. and Appl.} {\bf 36}(1998), 311--345.

\bibitem{sullivan} D.P. Sullivan, ``Travaux de Thurston sur les groupes quasi-fuchsiens et les vari\'et\'es hyperboliques de dimension $3$ fibr\'ees sur $S^1$,''
in {\em Bourbaki Seminar, Vol. 1979/1980}, {\em Lecture Notes in Math.}
{\bf 842}(1981), 196--214.

\bibitem{ThBook} W.P. Thurston, {\em The Geometry and Topology of
$3$-Manifolds}, Lecture Notes, Princeton University, 1979,
available at: 
\texttt{http://www.msri.org/publications/books/gt3m/}

\bibitem{thurston-earthquakes} W. P. Thurston, ``Earthquakes in two-dimensional hyperbolic geometry,''
in {\em Low-dimensional Topology and Klienian}, Cambridge
University Press, 1986, 91-112.


\bibitem{thurstonII} W. P. Thurston, ``Hyperbolic structures on $3$-manifolds, II:
surface groups and $3$-manifolds which fiber over the circle,'' preprint, available at:
\texttt{http://front.math.ucdavis.edu/math.GT/9801045}

 
}
 
\end{thebibliography}
\end{document}